\documentclass[12pt]{amsart}
\usepackage{amsmath, amsfonts, amssymb, amsthm}
\usepackage{mathtools}
\usepackage{caption}
\usepackage{subcaption}

\title[From Erd\H os--R\'enyi graphs to pure complexes ]{From Erd\H os--R\'enyi graphs to Linial-Meshulam complexes via the multineighbor construction}

\author{Eric Babson}

\address{Department of Mathematics,
              University of California at Davis,
              California 95616,
              U.S.A.}

\email{babson@math.ucdavis.edu}

\author{Jan Spali\'nski}

\address{Faculty of Mathematics and Information Science,
              Warsaw University of Technology,
              Koszykowa 75,
              00-662 Warsaw,
              Poland}
              
\email{jan.spalinski@pw.edu.pl}
              

\newtheorem{lemma}{Lemma}
\newtheorem{prop}{Proposition}
\newtheorem{thm}{Theorem}
\newtheorem{cor}{Corollary}
\newtheorem{conj}{Conjecture}
\theoremstyle{definition}\newtheorem{defin}{Definition}

\def\e#1{{\bf E}(#1)\ }

\def\bin#1#2{\textbf{Bin}_{#1,#2}}

\def\PP{\mathbb{P}}
\def\EE{\mathbb{E}}

\def\NN{\mathbb{N}}

\def\E{\mathbb{E\,}}
\def\e{\textrm{e}}
\def\var{\textrm{Var\,}}

\def\bigxf{X_{\mathcal F}}

\begin{document}

\begin{abstract} The {\it $m$-neighbor complex} of a graph is the simplicial complex in which faces are sets of vertices with at least $m$ common neighbors.  We consider these complexes for Erd\H os-R\'enyi random graphs and find that for certain explicit families of parameters the resulting complexes are with high probability $(t-1)$-dimensional with all $(t-2)$-faces and each $(t-1)$-face present with a fixed probability.  Unlike the Linial-Meshulam measure on the same complexes there can be correlations between pairs of $(t-1)$-faces but we conjecture that the two measures converge in total variation for certain parameter sequences.  
\end{abstract}


\maketitle

\section{Introduction} The {\it $m$-neighbor complex} $N_m(G)$ of a graph $G$ is the simplicial complex in which faces are sets of vertices with at least $m$ common neighbors. 
This construction is studied in detail in \cite{MS}.
We consider these complexes for Erd\H os-R\'enyi random graphs \cite{ER} and find that for certain explicit families of parameters the resulting complexes are with high probability $(t-1)$-dimensional with all $(t-2)$-faces and each $(t-1)$-face is present with a fixed probability.  Unlike the Linial-Meshulam (\cite{LM1}, \cite{ALLM}) measure on the same complexes there can be correlations between pairs of $(t-1)$-faces but we conjecture that the two measures converge in total variation for certain parameter sequences.  

\section{Preliminaries} We recall some well known facts and fix notation.

Write $[n]=\{1,2,\ldots, n\}$.  If $A$ is a finite set write $|A|$ for its cardinality, $\mathcal P A$ for the set of its subsets and ${A\choose c}\subseteq \mathcal P A$ for those with cardinality $c$.  If $X$ and $W$ are both graphs or both simplicial complexes write $Z_XW$ for the set of injective maps from $W$ to $X$, $z_xw=z_XW=|Z_XW|$ for their number which depends only on the shapes $x$ and $w$ of $X$ and $W$ as defined in section 4 below and if $Z_XW\not=\emptyset$ say that $X$ contains a copy of $W$.  If $G$ is a graph write $VG$ and $EG$ for its vertices and edges.   If $X$ is a simplicial complex write $\bigxf$ for the set of facets and $X_0$ for the set of vertices. 

A random variable $B$ is said to have a binomial distribution $B\sim \bin n  q$ if
\[        \mathbb P(B=k) = {n \choose k} q^k (1-q)^{(n-k)},\qquad k=0,\dots,n.                \]

The mean and variance are given by $\mu  = nq$ and $\sigma^2 = nq(1-q)$.
We will use the following bounds which are proven in the final section.   

{\bf Hoeffding's Inequalities:}\ 
{\it If $B\sim\bin n q$ then:
\begin{itemize}
\item $\PP(B\leq m)\leq \exp[-\frac{2}{n}(nq-m)^2]$ if $m<nq$ and 
\item $\PP(B\geq m)\leq \exp[-\frac{2}{n}(nq-m)^2]$ if $m>nq$.  
\end{itemize}
}

As an illustration, we include images of two small graphs and the 1-skeletons of the associated $1$- and $2$-neighbor complexes.

\begin{figure}[h!]
\centering
\includegraphics[width=1\linewidth]{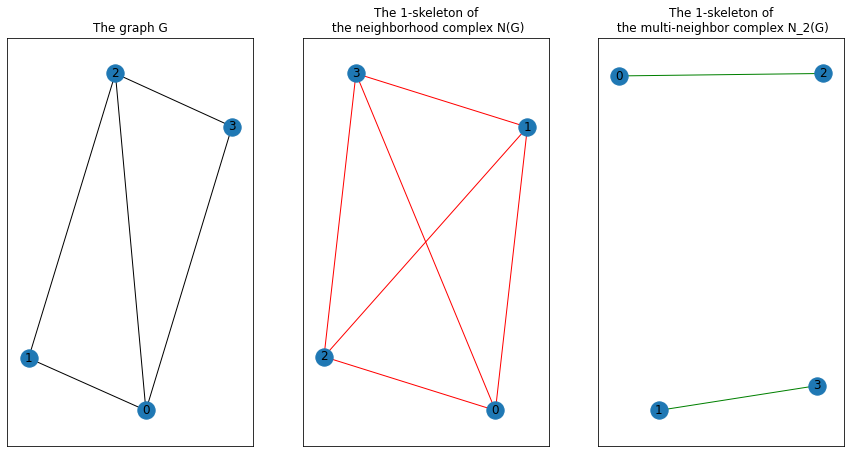}
\caption{A graph and the 1-skeleta of the $m$-neighbor complexes for $m=1$  and $m=2$.}
\end{figure}

\begin{figure}[h!]
\centering
\includegraphics[width=1\linewidth]{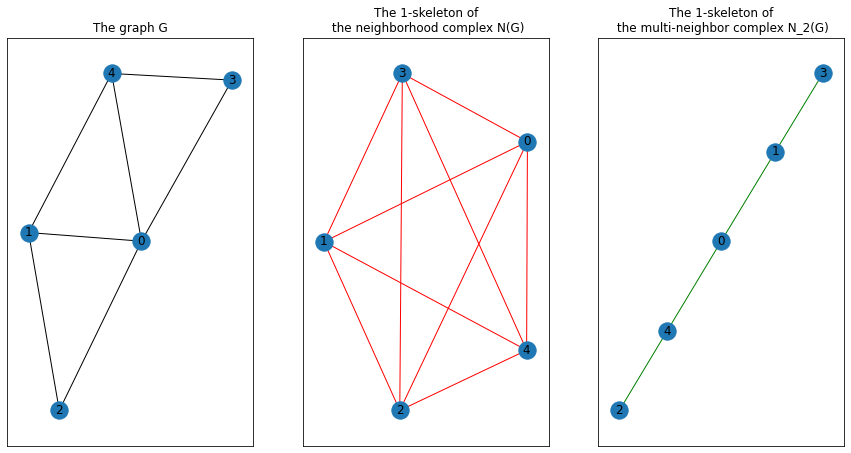}
\caption{A graph and the 1-skeleta of the $m$-neighbor complexes for $m=1$  and $m=2$.}
\end{figure}
\eject

\section{Support of $N_m(G(n,p))$}

Take $n$ to be a positive integer and $p\in (0,1)$ a probability and consider the Erd\H os--R\'enyi probability measure $G(n,p)$ on graphs with vertex set $[n]$ and each edge introduced independently with probability $p$. Consider $\Gamma_{n,m,p}=N_m G(n,p)$ the probability measure on simplicial complexes which is the $m$-neighbor complex of a random graph from $G(n,p)$.  Below is a picture of an Erd\H os--R\'enyi graph with parameters $n=100$ and $p=0.31$ on the left and the 1-skeleton of its $14$-neighbor complex  on the right.  

\begin{figure}[h!]
\centering
\includegraphics[width=1\linewidth]{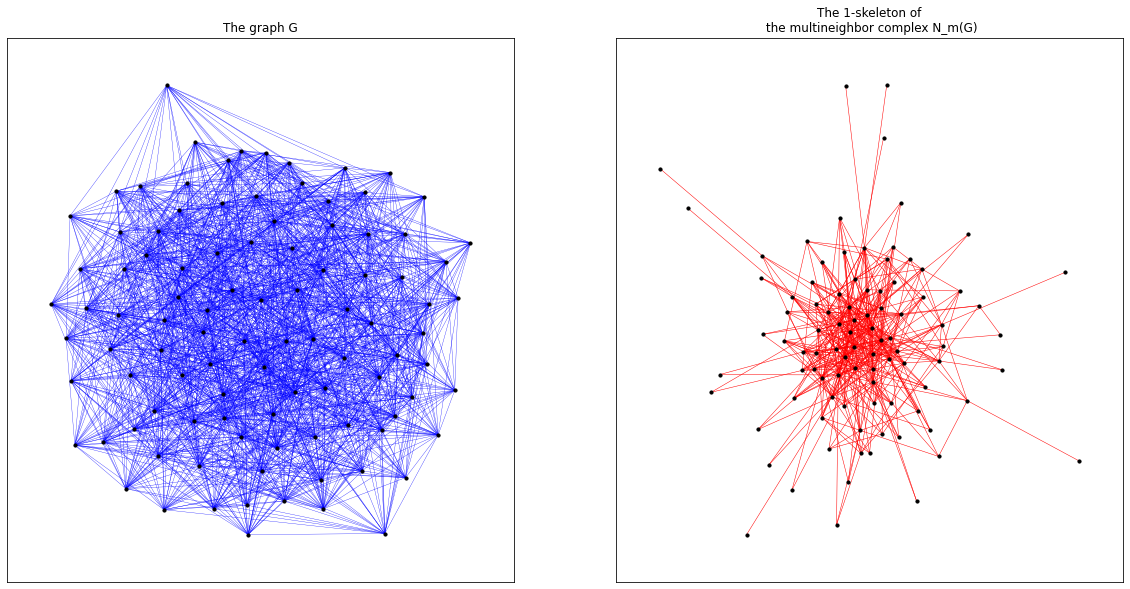}
\caption{An Erd\H os--R\'enyi graph from G(100,0.31) and the 1-skeleton of the 14-neighbor complex.}
\end{figure}

Given $n$, $m$ and $p$ let $t$ be the number defined as follows:
\[                 t = \left[\!\left[   \frac{\ln(n)-\ln(m)}{-\ln(p)}     \right]\!\right]            = \left[\!\left[   \log_p\left(\frac{m}{n}\right)\right]\!\right]                    \]
with $[[\cdot]]$ meaning a closest integer, taking the smaller choice if there are two possibilities.   Next,  take $\tau$ with $|\tau|\leq\frac{1}{2}$ to be the difference between $t$ and the expression being rounded:
\[         \tau =    t  +   \frac{\ln(n)-\ln(m)}{\ln(p)}     = t + \log_p \left( \frac{n}{m}  \right)               \]
so that
\[    p^t = p^{ \log_p \left( \frac{m}{n}  \right) +\tau }  = \left(\frac{m}{n}\right)p^\tau.     \]

Let $Y_{n,k-1}$ be the set of simplicial complexes with:
\begin{itemize}
\item vertex set $[n]$, 
\item all faces with $k-1$ vertices and
\item no faces with $k+1$ vertices.
\end{itemize}
Hence all but one complex in $Y_{n,k-1}$ is $(k-1)$-dimensional. 
Let $Y_{n,k-1,q}$ be the Linial-Meshulam probability distribution on $Y_{n,k-1}$ so that:  
\begin{itemize}
\item faces with $k$ vertices occur independently with probability $q$.
\end{itemize}
We show that for a large range of choices of $m$ and $n$, with high probability the complexes drawn from $\Gamma_{n,m,p}$ belong to $Y_{n,t-1}$ with $t=\left[\!\left[   \log_p\left(\frac{m}{n}\right)\right]\!\right] $ as above and further the probability that such a complex contains any particular $(t-1)$-face is the probability that $B\sim\bin {n-t} {p^t} $ takes a value of at least $m$.  Call this face probability $q$.  The distributions $\Gamma_{n,m,p}$ and $Y_{n,t-1,q}$ differ in that in $\Gamma$ face occurances may be correlated while in $Y$ they are not.    

\begin{lemma} If $p\in (0,1)$ there is $c>0$ (with $c=\frac{1}{2}$ if $p\le\frac{1}{4}$) so that for any $n$, $m$ and $f\in{[n]\choose t+1}$ with $t=\left[\!\left[\log_p\left(\frac{m}{n}\right)\right]\!\right]$ as above $\PP_{K\in \Gamma_{n,m,p}}(f\in K)\leq\exp\left[ -c\frac{m^2}{n} \right]$.
\end{lemma}

\begin{proof}  If $\Gamma$ is a graph with $V\Gamma=[n]$ and $f\in{[n]\choose t+1}$ write $\beta_f\Gamma=|\{v\in V\Gamma-f| \{v\}\times f\subseteq E\Gamma\}|$ for the number of common neighbors so $\beta_fG(n,p)=B\sim \bin {n-t-1}{p^{t+1}}$ and $\PP_{K\in \Gamma_{n,m,p}}(f\in K)=\PP(B\geq m)$.
First consider the case $p\leq\frac{1}{4}$ and take $c=\frac{1}{2}$. 
  In order to apply Hoeffding's inequality, we take $|\tau|\leq\frac{1}{2}$ as above so $p^{t} = \frac{m}{n} p^{\tau}$ and verify that if $|f|=t+1$ then $\mu=\E B<m$: 
\[    \mu=(n-(t+1))p^{t+1} = (n-(t+1))\cdot\frac{m}{n}\cdot p^{\tau+1} = m \cdot\frac{n-(t+1)}{n}\cdot p^{\tau+1} < m.       \]

Moreover, since $p\in(0,\frac{1}{4}]$, we have
\[        0\le  \frac{m}{n} p^{\tau+1}\le \frac{1}{2}\left( \frac{m}{n-(t+1)} \right).         \]

Hence we have:

\begin{align*}
\PP_{K\in \Gamma_{n,m,p}}(f\in K) &= \PP(B\ge m) \\
                                                &\le \exp\left[-2(n-t-1)\left(p^{t+1}-\frac{m}{n-(t+1)}\right)^2\right]      \\
                                                &\le \exp\left[-2(n-t-1)\frac{1}{4}\left(\frac{m}{n-(t+1)}\right)^2\right]   \\
                                                &\le \exp\left[-\frac{1}{2}\left(\frac{m^2}{n-(t+1)}\right)\right]       \\
                                                &\le \exp\left[-\frac{1}{2}\left(\frac{m^2}{n}\right)\right].       \\
\end{align*}

Finally for the case $p>\frac{1}{4}$ take $c=2(1-p^{\frac{1}{2}})^2$.
The argument is analogous to the $p\leq\frac{1}{4}$ case upon noting that $0<p^{\tau+1}<p^{\frac{1}{2}}$ and $\frac{m}{n-(t+1)} \ge \frac{m}{n}$ and hence
\[         \frac{m}{n-(t+1)} - \frac{m}{n} p^{\tau+1} \ge \frac{m}{n-(t+1)} - \frac{m}{n-(t+1)} p^{\tau+1}  \ge\frac{m}{n-(t+1)}  \left( 1-p^{\frac{1}{2}} \right).               \] 

\end{proof}

\begin{lemma} If $p\in (0,1)$ there is $c>0$ (with $c=\frac{1}{3}$ if $p\le\frac{1}{4}$) so that for any $n\ge 9$, $m$ and $f\in{[n]\choose t-1}$ with $t=\left[\!\left[\log_p\left(\frac{m}{n}\right)\right]\!\right]$ as above $\PP_{K\in \Gamma_{n,m,p}}(f\not\in K)\leq\exp\left[ -c\frac{m^2}{n} \right]$.
\end{lemma}

\begin{proof}  Similarly to the previous proof  if $f\in{[n]\choose t-1}$  then $\PP_{K\in \Gamma_{n,m,p}}(f\not\in K)=\PP(B<m)$ with $B\sim \bin {n-(t-1)}{p^{t-1}}$.
Once again consider first the case $p\leq\frac{1}{4}$ and take $c=\frac{1}{3}$.  
  Note that
\[  t-1 =  \left[\!\left[   \frac{\ln(n)-\ln(m)}{-\ln(p)}    \right]\!\right]   -1        \le         \frac{\ln(n)}{\ln(p^{-1})} \le \frac{\ln(n)}{\ln 4} \le \ln(n) \le \sqrt{n}.   \]
Hence
\begin{equation}
    \frac{n}{n-(t-1)}   \le  \frac{n}{n-\sqrt{n}}    \le   \frac{n(n+\sqrt{n})}{n^2-n} \le \frac{n+\sqrt{n}}{n-1}.   \label{lemma_two_inequality}   
\end{equation}
The function on the right hand side above is clearly decreasing (for $n>1$), and has value  $\frac{3}{2}$ for $n=9$, hence the left hand side above is bounded
by $\frac{3}{2}$ for all $n\ge 9$. Moreover with $|\tau|\leq\frac{1}{2}$ as above, $p^{\tau-1}\ge 2$. Hence we have
\begin{align*}
p^{\tau-1} - \frac{n}{n-(t-1)} &\ge \frac{1}{2}, \\
\frac{p^{\tau-1}}{n} - \frac{1}{n-(t-1)} &\ge \frac{1}{2n}. \\
\end{align*}

In order to apply Hoeffding's inequality, we verify that $\mu=\E B>m$. 
\[    \mu=(n-(t-1))p^{t-1} = (n-(t-1))\cdot\frac{m}{n}\cdot p^{\tau-1} = m \cdot\frac{n-(t-1)}{n}\cdot p^{\tau-1} > m,       \]
where the last inequality follows from the estimates of the factors in the previous paragraph.
Hence we have.

\begin{align*}
  \PP_{K\in\Gamma_{n,m,p}}(f\not\in K) & = \PP(B\le m-1) \\
  & \le \PP(B\le m) \\
                                                                              &\le \exp\left[-2(n-(t-1))\left(p^{t-1}-\frac{m}{n-(t-1)}\right)^2\right]      \\
                                                                              &= \exp\left[-2(n-(t-1))\left(\frac{m}{n} p^{\tau-1}-\frac{m}{n-(t-1)}\right)^2\right]      \\
                                                                              &= \exp\left[-2m^2(n-(t-1))\left(\frac{1}{n} p^{\tau-1}-\frac{1}{n-(t-1)}\right)^2\right]      \\
                                                                              &\le \exp\left[-2m^2(n-(t-1))\frac{1}{4n^2}\right]      \\
                                                                              &\le \exp\left[-\frac{1}{3} \frac{m^2}{n}\right].      \\
\end{align*}
The last inequality follows from the fact that $\frac{n-(t-1)}{n} \ge \frac{2}{3}$.

Finally for the case $p>\frac{1}{4}$ take $c=\frac{1}{3}(p^{-\frac{1}{2}}-1)^2$.

Choose $\varepsilon=p^{-\frac{1}{2}}-1$ so that
\[          p^{d-1} \ge \frac{1}{\sqrt{p}} = 1 + \varepsilon.        \]
From (\ref{lemma_two_inequality}) there exists an $N$ such that for $n>N$ we have 
\[          \frac{n}{n-(t-1)} \le 1+\frac{\varepsilon}{2}.        \]
Hence for $n>N$ we have
\begin{align*}
p^{t-1} - \frac{m}{n-(t-1)} &= \frac{m}{n} p^{d-1} - \frac{m}{n-(t-1)} \\
                                         &= \frac{m}{n} \left(  p^{d-1} - \frac{n}{n-(t-1)}  \right) \\
                                         &\ge \frac{m}{n}  \cdot  \frac{\varepsilon}{2}.
\end{align*}

Using this and $n-(t-1)\ge \frac{2}{3}n$ an argument analogus to the $p\leq\frac{1}{4}$ case yields (for $n>N$):
\[  \PP(B< m) \le \exp\left[-\frac{\varepsilon^2}{3}\left(\frac{m^2}{n}\right)\right].   \]

\end{proof}

For the example in figure 3 with $n=100$ and $p=.31$ there is $t=2$, $\tau=.32$, $c=.39$ for lemma 1 and $c=.21$ for lemma 2.  Thus lemma 1 implies that the chance of each triple of vertices of the graph to not be a triangle in the complex is at least $.53$ so the chance that there are no triangles is at least $0$ as these events are not independent while lemma 2 implies that the chance of each vertex of the graph to be a vertex of the complex is at least $.66$ so the chance that they all are is at least $10^{-18}$ which is nonzero only because these events are independent.   Thus these parameters are not in the regime addressed in the following theorem with $t=3$ in which the complexes are gauranteed to have high probability of having every vertex of the graph as a vertex and no triangles.  

The first two parts of the following theorem use these two lemmas while the third part follows from lemmas 4 and 5 in the next section.  

\begin{thm} If $p_n\in(0,1)$ and $m_n\in \NN$ are sequences for which 
any of the following three pairs of conditions holds:
\begin{itemize}
\item $p_n$ is constant and $\lim_{n\to\infty} \frac{m_n^2}{n(\ln n)^2}=\infty$,
\item $\lim_{n\to\infty}p_n=0$ and $\lim_{n\to\infty}\frac{-(\ln p) m_n^2}{n(\ln n)^2}>4$ or
\item $m_n=m$ is constant and there constants $t<\sqrt{2m+1}$ an integer and $b\in (t-1,\frac{m(t+1)}{m+t+1})$ with $p_n=n^{\frac{-1}{b}}$.  
\end{itemize}
then
  \[      \lim_{n\to\infty} \PP_{K\in\Gamma}(K\in Y) = 1               \]
where $\Gamma=\Gamma_{n,m_n,p_n}$ and $Y=Y_{n,t_n-1}$ with $t_n=\left[\!\left[\log_{p_n}\left(\frac{m_n}{n}\right)\right]\!\right]$.  
\end{thm}

\begin{defin}
  Write {\em $\Gamma_m$ has property $P$ asymptotically almost surely (aas)} if $ \lim_{n\to\infty} \PP_{K\in\Gamma}(K\hbox{ has property }P) = 1$. \end{defin}

Thus the conclusion of the theorem is that $\Gamma\in Y$ aas.  

Figures \ref{fig:First_Conditions} and \ref{fig:Second_Conditions} ilustrate the first two sets of coditions, with the value of $t=2$, the size $n$ of the Erd\H os--Renyi graph given on the horizontal axis and the proporiton of simplices in the $m$-neighbor complex to the maximal possible displayed on the vertical axis. 

\begin{figure}[h!]
\centering
\includegraphics[width=0.6\linewidth]{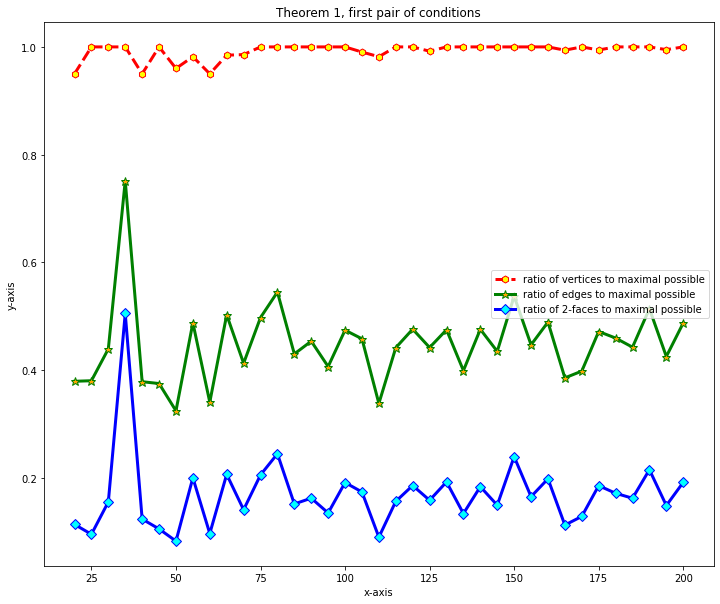}
\caption{First set of conditions. Here $p=0.5$, $m=\textrm{Ceiling\,}(n/4)$ and $t=2$. Here $q\approx 0.49$.}
\label{fig:First_Conditions}
\end{figure}

\begin{figure}[h!]
\centering
\includegraphics[width=0.6\linewidth]{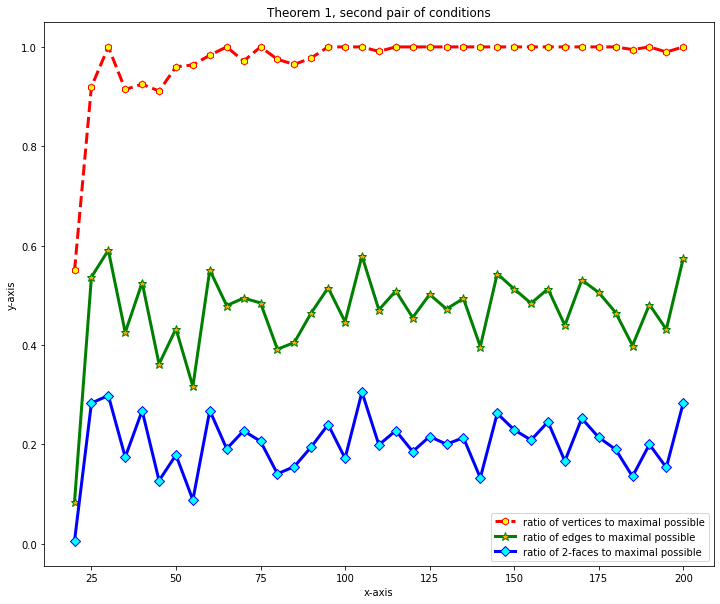}
\caption{Second set of conditions. Here $p=1/ln(ln(n))$, $m=\textrm{round}(n p^2)$ and $t=2$. Here $q\approx 0.48$.}
\label{fig:Second_Conditions}
\end{figure}

\begin{proof} This is a proof of the first two parts.  The third follows immediately from lemmas 4 and 5 of the next section.  
\newline   First use the second lemma above and the first moment method to see that ${[n]\choose t_n-1}\subseteq \Gamma$ aas.
  
 Let $K$ be a complex drawn from $\Gamma$ as described in the statement of the theorem, and let $N$ be the random variable
counting the number of $(t_n-1)$-element subsets of $K_0=[n]$ which are not faces of $K$. By the First Moment Method (see \cite{FK}, 
Lemma 22.2), we have 
\[\PP(N>0)\le \E N.\] 
Let $\kappa_n = \frac{m_n^2(-\ln p_n)}{n(\ln n)^2}$. By the second lemma above there is a constant $c>0$ depending on $p$ with

\begin{align*}
\EE N &= {n\choose {t_n-1}} \PP_{K\in\Gamma(n,m,p)}\left(\textrm{a fixed $(t_n-1)$-tuple is not a face of $K$}\right) \\
       &\le  n^{t_n-1} \exp\left[ -c\frac{m_n^2}{n} \right] \\
       &\le  \exp\left[\left(t_n-\frac{1}{2}\right) \ln(n)\right] \exp\left[ -c\frac{m_n^2}{n} \right]    \\
        &\le  \exp\left[\log_{p_n}\left(\frac{m_n}{n}\right)\ln(n)  -c\frac{m_n^2}{n} \right]    \\
          &\le  \exp\left[\frac{(\ln n)^2}{-\ln p_n}\left(1- c\kappa_n \right)\right].    
\end{align*}

In the first case of the theorem $\lim \kappa_n = \infty$ and $c$ depends only on $p$ so the limit of the last expression is equal to zero.

In the second case of the theorem $\lim \kappa_n \geq 2$ and $c=\frac{1}{2}$ so again the limit of the last expression is equal to zero.

Next use the first lemma above and the first moment method again to show that $\Gamma$ aas has dimension at most $t$.
The argument is very similar to that above but also uses the bound 
\[     t_n+1 \le     2 \frac{\ln(n)-\ln(m_n)}{-\ln(p_n)}.     \]
\end{proof}

\section{Asymptotics of $N_m(G(n,p))$}

In this section we consider the hypotheses from the third part of theorem 1.  That is $\Gamma=\Gamma_{n,m,p}$ with $p=n^{\frac{-1}{\beta}}$ for a fixed density parameter $\beta>0$, a fixed number $m$ and a growing number $n$ of vertices.   This makes the parameter $t-\tau$ from the previous section converge to $\beta$.  
We then fix a finite simplicial complex $X$ and study the limiting probability that $X$ is isomorphic to a subcomplex of a complex $K$ chosen from $\Gamma$.   This analysis includes a proof of the last part of theorem 1 but does not give total variation convergence which we conjecture below.  

\begin{defin}
Call $\beta$ a {\em threshold} for a property of a complex in $\Gamma_m$ if a complex drawn from $\Gamma_{m,n,p}$ with $p=n^{\frac{-1}{b}}$ has the property aas as $n$ grows if $b>\beta$ and does not have it aas as $n$ grows if $b<\beta$.  This ignores the behavior at $b=\beta$.

Similarly, call $\beta$ a {\em threshold} for a property of a graph $H$ drawn from $G(n,p)$, where $p=n^{\frac{-1}{b}}$ , 
if $H$ aas has the property as $n$ grows if  $b>\beta$ and does not have it aas as $n$ grows if $b<\beta$.
\end{defin}

The second part of the above definition is consistent with Definition 1.6 in \cite{FK}, for the choice of the threshold function
$p^*(n) = n^{\frac{-1}{\beta}}$.

The first part of this section sets up notation to define the $m$-density of a complex $X$ and shows that it is a threshold for $\Gamma_m$ to contain $X$ as a subcomplex.  

\begin{figure}[h!]
    \centering
    \begin{subfigure}{.32\textwidth}
        \includegraphics[width=1.14\linewidth]{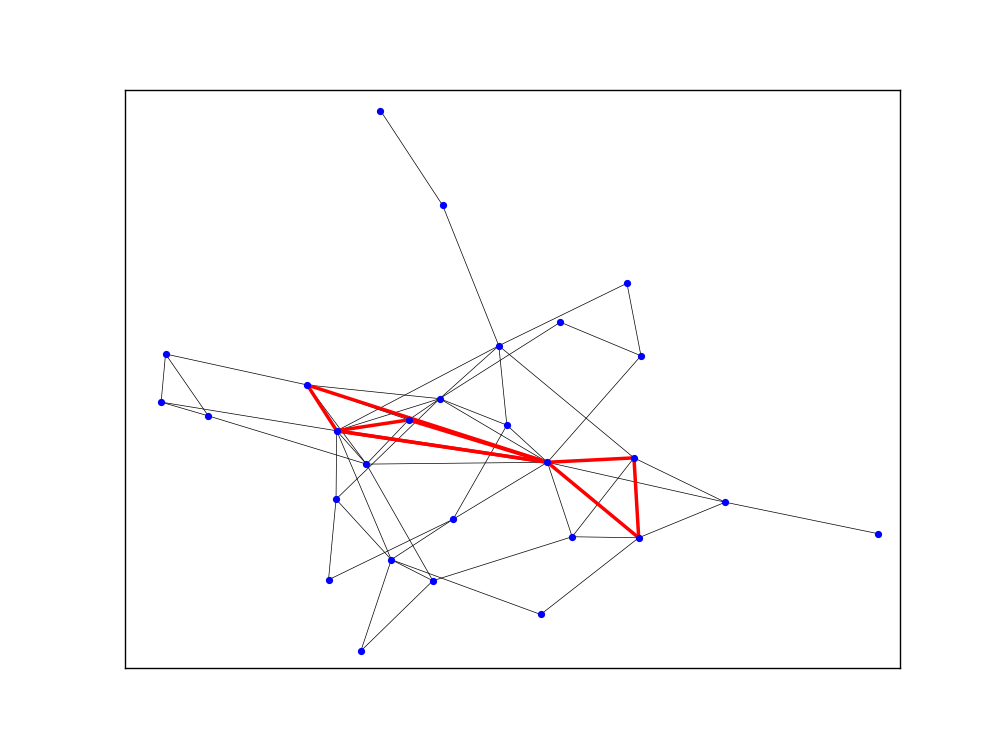}
    \end{subfigure}
    \begin{subfigure}{.32\textwidth}
        \includegraphics[width=1.14\linewidth]{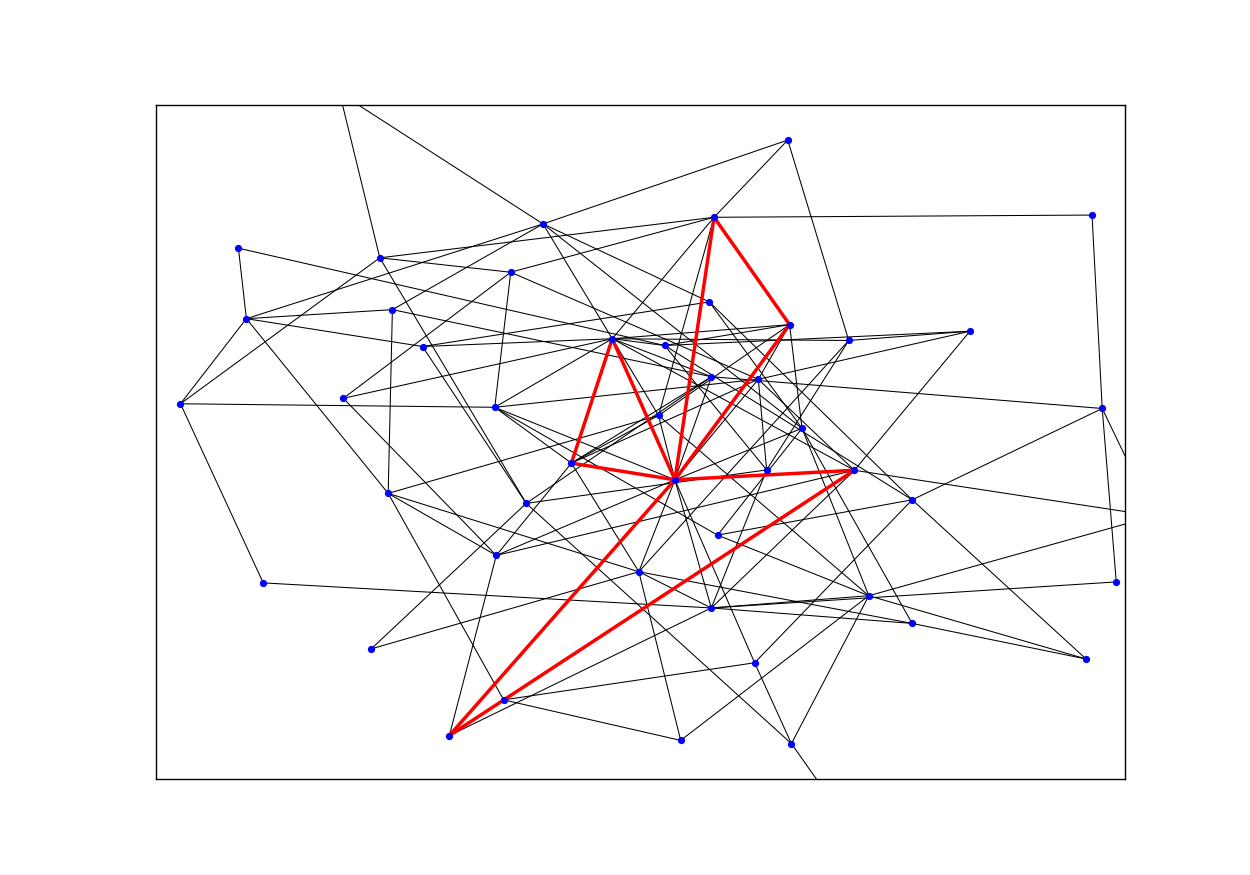}
    \end{subfigure}
    \begin{subfigure}{.32\textwidth}
        \includegraphics[width=1.14\linewidth]{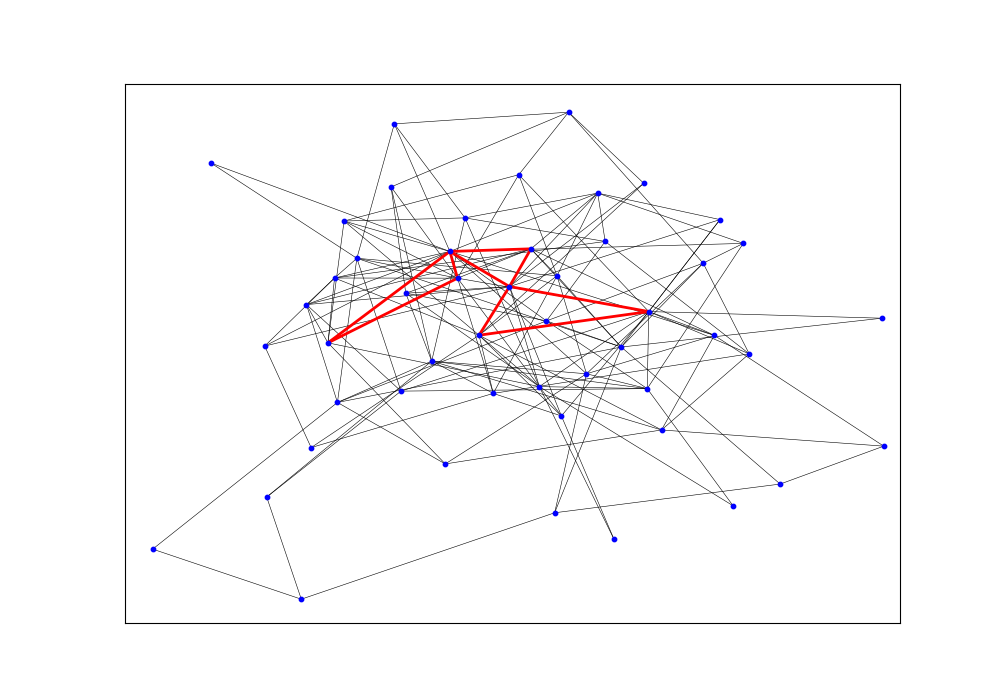}
    \end{subfigure}
    \caption{Copies of pure simplicial complexes in m-neighbor constructions on Erd\H os-- R\'enyi graphs}
    \label{fig:Correlations}
\end{figure}

The result is based on one for subgraphs of random graphs which Erd\H os and R\'enyi proved for balanced graphs (see below) and Bollob\'as stated in the form we use.  The account given in Frieze--Karo\'nski \cite{FK}
is particularly useful for our purposes.

Define the density of a nonempty graph $H$ as the ratio of the number of edges to the number of vertices:
\[             d_H=\frac{e_H}{v_H}       \]
and the related maximum subgraph density:
\[                 \bar{d}_H = \max\{ d_K: \emptyset\not=K\subseteq H     \}.       \]
A graph is balanced if $\bar{d}_G=d_G$ and strictly balanced if $d_G>d_H$ for all proper subgraphs $H$ in $G$.

{\bf Theorem 5.3 in \cite{FK}} If $H$ is a graph with $d_H>0$, then $\bar{d_H}$ is a threshold for the appearence of 
$H$ in $G(n,p)$ with $p=n^{\frac{-1}{b}}$.
\medskip


To study the probability of finding a copy of a finite complex $X$ with facets $F=\bigxf$ in a complex drawn from $\Gamma_{n,m,p}$, we will consider  functions 
\[       W: F \rightarrow \mathcal \mathcal P X_0 \]
along with the functions they induce on the power set
\[       W^{\cap},  W^{!}, W^{\cup}: \mathcal PF \rightarrow    \mathcal P X_0\   \]
which take a set of facets respectively to the intersection, exclusive intersection or union of the images of $W$.  Write $W^*_A=W^*(A)$, $w^*_A=|W^*_A|$ and by convention $W^\cap_\emptyset=X_0$.  Call each $W^*$ a version of the $F$-set $W$ and each $w^*$ a version of the $F$-shape $w$.  

{\bf Example 1.} Let $X$ be the pure $(3-1)$ dimensional simplicial complex with facets $F=\{f,g,h\}$ where
\[f=\{\alpha,\gamma,\delta\}, \quad g=\{\gamma,\delta,\theta\}, \quad\textrm{and}\quad  h=\{\delta,\kappa, \lambda\}.\] 
The complex and the geometric realization are displayed in Figure \ref{fig:Pure_Complex_1}.

\begin{figure}[h!]
    \centering
    \begin{subfigure}{.45\textwidth}
	\includegraphics[width=1.7\linewidth]{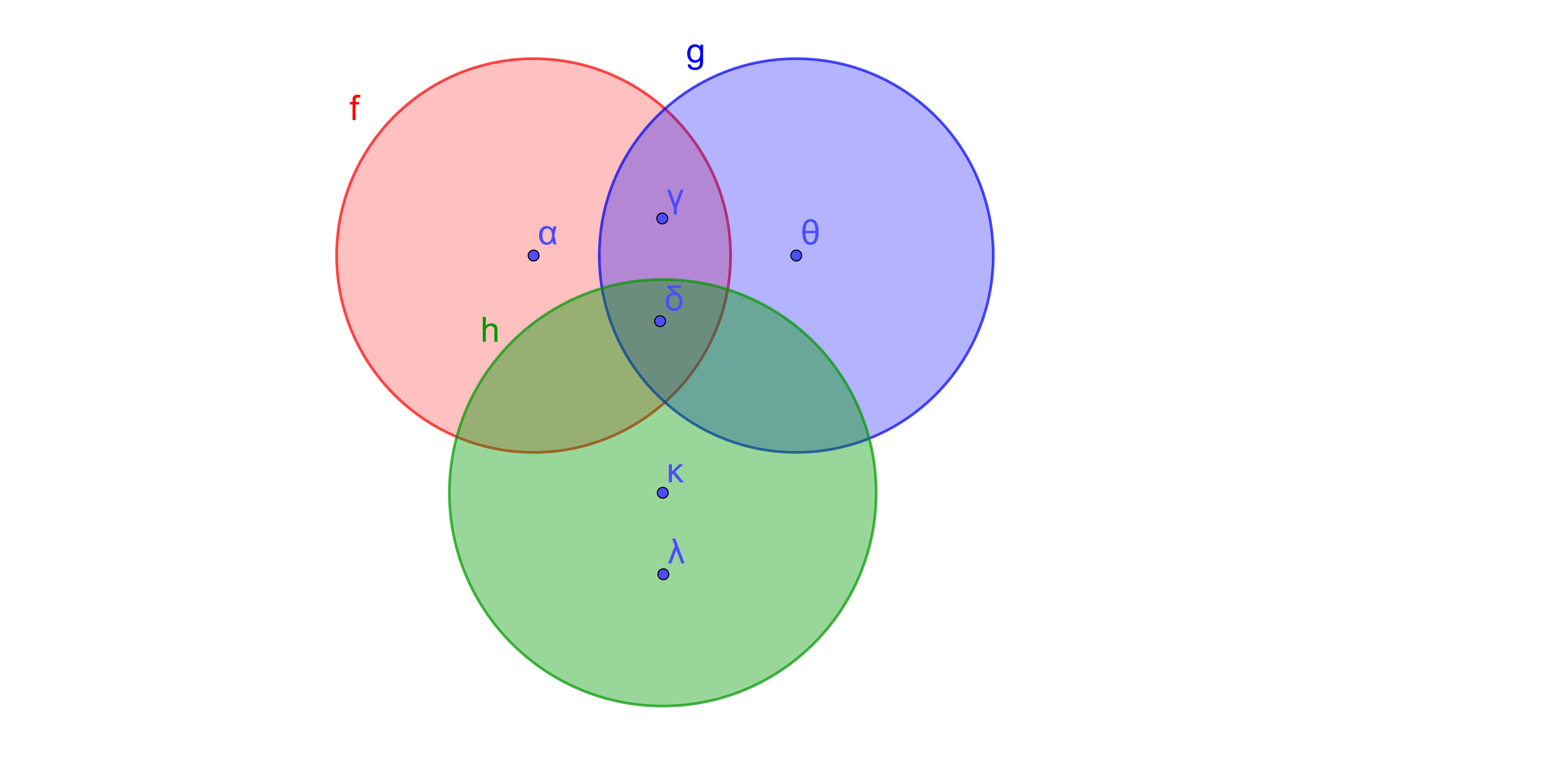}
  	\end{subfigure}
    	\begin{subfigure}{.45\textwidth}
	\includegraphics[width=1.7\linewidth]{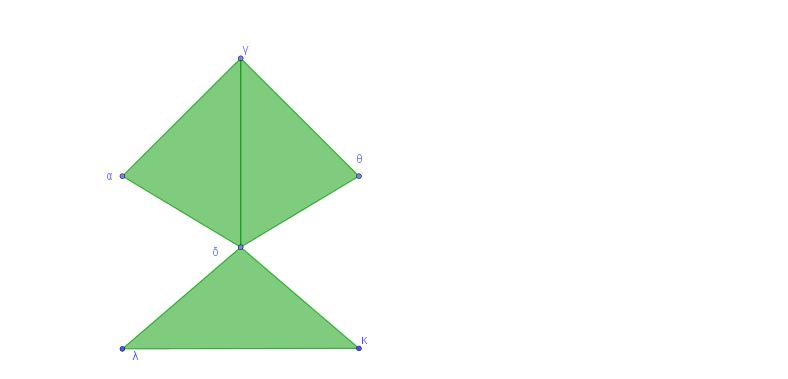}
  	\end{subfigure}
    \caption{The pure simplicial complex $X$ and its geometric realization}
    \label{fig:Pure_Complex_1}
\end{figure}
	
Hence 
\[\mathcal P F =\{ \emptyset, \  \{f\},\   \{g\},\  \{h\},\  \{f,g\},\  \{f,h\},\  \{g,h\},\  \{f,g,h\}\}\] 
and $X$ is itself an $F$-set with shape $x$ with the values of $x^{\cap }$, $x^{!}$ and  $x^{\cup}$
on the subsets $A$ of $F$ given in the following table.

\bigskip

\renewcommand{\arraystretch}{1.4}
\begin{tabular}{|c |c |c| c| c| c| c|c|c|}
\hline
$A$   & $\emptyset$     &    $ \{f\}$    &   $\{g\}$    &  $ \{h\}$    &   $\{f,g\}$   &    $\{f,h\}$     &  $ \{g,h\} $   &   $\{f,g,h\}$       \\
\hline
$x^{\cap }_{A}$   & $6$     &    $ 3$    &   $3$    &  $3$    &   $2$   &    $1$     &  $ 1 $   &   $1$       \\
\hline
$x^{!}_{A}$  & $0$     &    $1$    &   $1$    &  $2$    &   $1$   &    $0$     &  $ 0 $   &   $1$       \\
\hline
$x^{\cup }_{A}$   & $0$     &    $3$    &   $3$    &  $3$    &   $4$   &    $5$     &  $5$   &   $6$       \\
\hline
\end{tabular}

\rightline{$\square$}
\bigskip

Note that any one of the three  versions $x^{\cap}$, $x^{!}$ and $x^{\cup}$ can be explicity expressed in terms of any other 
as described in the following proposition.

\begin{prop} If $x$ is an $F$-shape and $A \subseteq F$ then:

\begin{align*}
&(a)\quad x^{\cup}_{A}=\sum_{\emptyset\not= B\subseteq A}(-1)^{|B|-1} x^{\cap}_{B}\\
&(b)\quad x^{!}_{A}=\sum_{B\supseteq A}(-1)^{|B|-|A|} x^{\cap}_{B}\\
&(c)\quad x^{\cup}_{A}=\sum_{B\cap A\not=\emptyset} x^{!}_{B} \\
&(d)\quad x^{\cap}_{A}=\sum_{B\supseteq A} x^{!}_{B}\\
&(e)\quad x^{\cap}_{A} = \left.\begin{cases}\displaystyle\sum_{ B\subseteq A}    (-1)^{|B|-1}   x^{\cup}_{B}, \qquad &A\neq \emptyset, \\
                                                                          \quad  x^{\cup}_{F} \qquad &A =\emptyset\end{cases} \right. \\
&(f)\quad x^{!}_{A}=\sum_{B\subseteq A\neq \emptyset} (-1)^{|A|-|B|+1} x^{\cup}_{F \setminus B}.
\end{align*}

\end{prop}

\begin{proof}
Formulas $(a)$ and $(b)$ follow directly from the inclusion and exclusion formula (see M. Aigner \cite{Aigner}, Chapter 5, Sieve Methods, Section 1: Inclusion-Exclusion).  
Formula $(c)$ follows from the fact that each element of $X_0$ increases $ t_A^{!}$ by $1$ for a unique $A\subseteq F$. 
\end{proof}

More generally a triple $ x=( x^{\cap}, x^!, x^{\cup})$ of integer valued functions on $\mathcal PF$ related via the above summations is called an $F${\it -shape} while a set valued one is called an $F${\it -set}.  The three entries $x^*$ are called versions of $x$.    
Write $x_0=x^{\cap}_{\emptyset}=x^{\cup}_F$.  If $x$ is the shape of a simplicial complex then $x_0$ is the number of vertices.  A final useful construction is the pointwise $\cap$-product of $F$-shapes defined by $(wx)^{\cap}_{A}=w^{\cap}_{A}x^{\cap}_{A}$.  Note that it is the $\cap$-versions which multiply pointwise while the effects on the $!$- and $\cup$-versions are more complicated.


Call an $F$-shape $x$ {\em nonnegative} if $x^!_A\geq 0$ for every $A\subseteq F$ and {\em $k$-pure} ({\em pure})  if $x^\cap_{\{f\}}=k$ for every $f\in F$.  In the latter case write $\bar{x}=k$.  Note that the shape of any simplicial complex is nonnegative and the shape is $k$-pure exactly if the complex is $(k-1)$-pure.  
Given $F$-shapes $z$ and  $x$ we say that $z\le x$ if $x-z$ is nonnegative. 

A key quantity for our considerations is a measure of density required for the 
appearence of $X$ in $\Gamma_{n,m,p}$.
\begin{defin}[$m$-density of $X$]
For a pair of pure $F$-shapes $x$ and $w$, let   $b(x,w)=\frac{(xw)_0}{x_0+w_0}$.
We define the {\em $m$-density of $X$} as
\[b_m(x) = \min_{\bar{w}=m} \left\{\max_{  z, v > 0 }\left\{ b(z,v)  \mid    z\le x,\,  v\le w   \right\} \right\} \]

Write also $b_m(X)=b_m(x)$ if $X$ has shape $x$.  
\end{defin}

These arise in the next theorem when studying whether a complex $K$ drawn from $\Gamma_{n,m,p}$ contains a copy of a given pure finite simplicial complex $X$ with shape $x$ by considering an $m$-pure $F$-set $W$ with shape $w$.  Write $H=H(G,X,W)$ for the set of all injective maps $\rho:X_0\cup W_0\rightarrow VG$ for which $\cup_{f\in F}\left[\rho(X_{\{f\}}^\cap)\times \rho(W_{\{f\}}^\cap)\right]\subseteq EG$.  Thus if $\rho\in H$ then $\rho|_{X_0}:X\rightarrow N_mG$ induces an injective map of simplicial complexes.  If $G$ is drawn from $G(n,p)$ then the log base $n$ of the expected value of  $|H|$ is positive if $\beta>b(x,w)$ for $n$ sufficiently large, as the following calculation shows. 
\begin{align*}
\lim_{n\to\infty}\log_n \EE |H| &= \lim_{n\to\infty} \log_n {n \choose {x_0+w_0}} p^{(xw)_0} \\
                                                &= x_0+w_0 -\frac{b}{\beta}(x+w)_0 = (x_0+w_0)\left(1-\frac{b}{\beta}\right).                       
\end{align*} 

\medskip 

{\bf Example 2.} Consider the pure (3-1)-dimensional simplicial complex $X$ of shape $x$ with facets $F=\{f,g,h\}$ where: 

\[       f=\{\alpha,\gamma,\delta \},   g=\{ \gamma,\delta, \theta  \}, h=\{\theta, \kappa, \lambda\}                   \]
so $x_0=6$, $\bar{x}=3$, $x_F=3$ and $\phi_x=\frac{3}{2}$.  

The complex and the geometric realization are displayed in Figure \ref{fig:Pure_Complex_2}.

\begin{figure}[h!]
    \centering
    \begin{subfigure}{.45\textwidth}
	\includegraphics[width=1.4\linewidth]{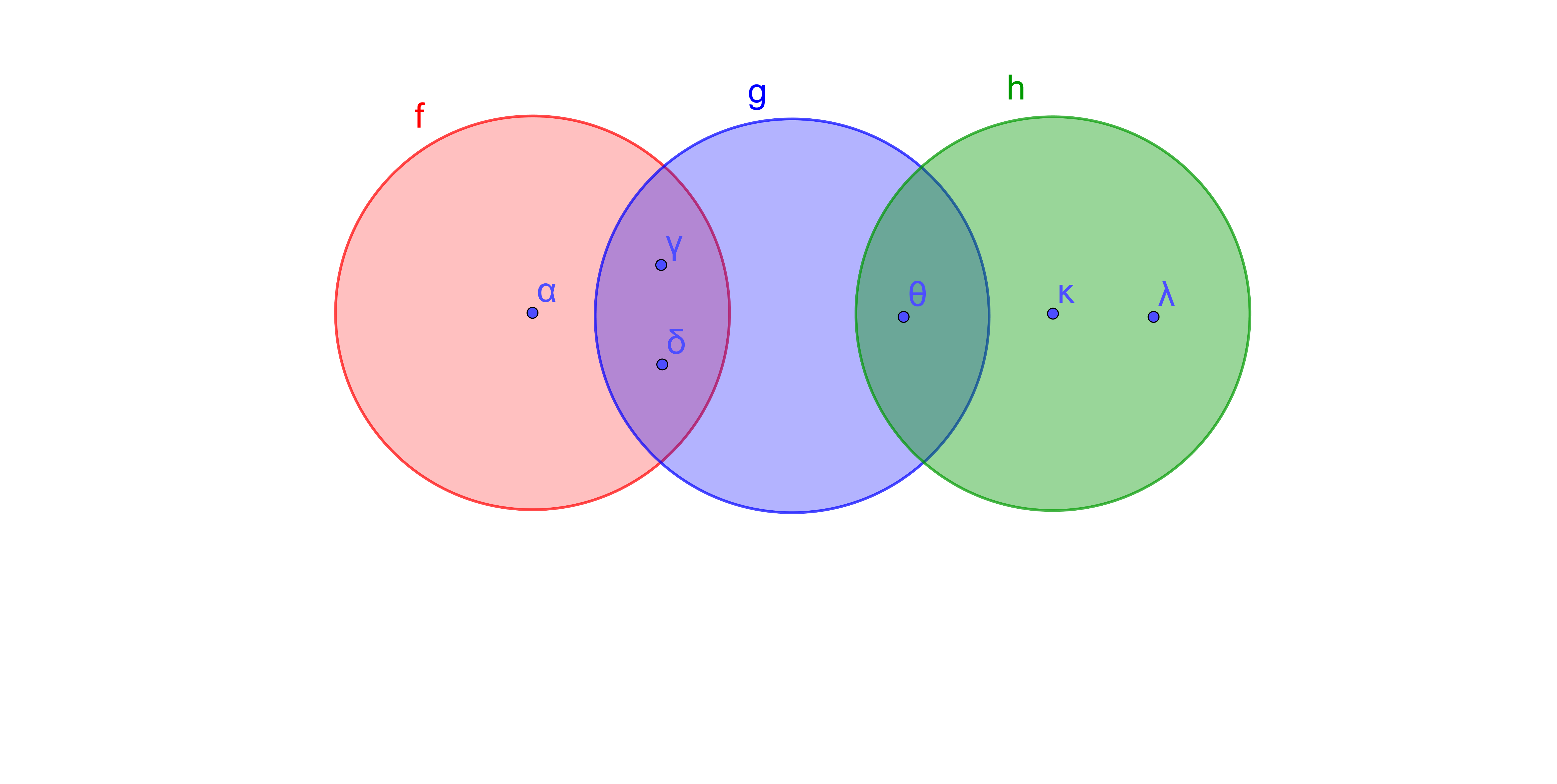}
  	\end{subfigure}
    	\begin{subfigure}{.45\textwidth}
	\includegraphics[width=1.4\linewidth]{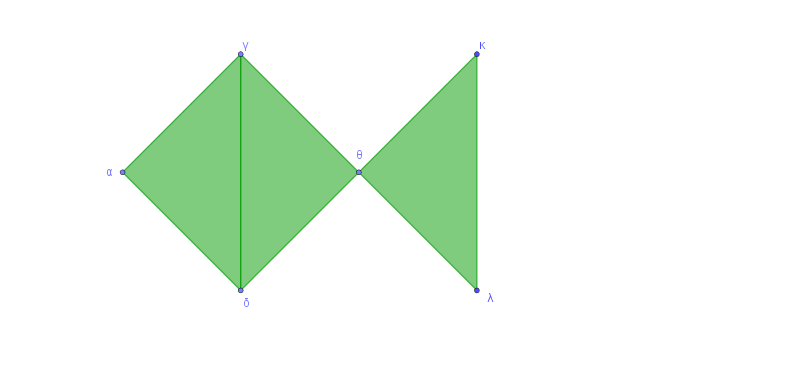}
  	\end{subfigure}
    \caption{The pure simplicial complex $X$ and its geometric realization}
    \label{fig:Pure_Complex_2}
\end{figure}
	
If a copy of the complex $X$ appears in a complex $K=N_2G$ drawn from $\Gamma_{n,2,p}$ via $\rho:X_0\rightarrow VG=K_0$ then 
\begin{itemize}
\item $X$ and $\rho X$ are $F$-sets with the same shape and
\item for each facet $f\in F$ the associated vertices $\rho X^{\cup}_{\{f\}}\subseteq VG$ have at least two common neighbors in $G$.  
\end{itemize}
Choose any such pair to be $W^{\cup}_{\{f\}}$  and call the resulting $F$-set $W$ (which by construction is $2$-pure) a {\it $2$-witness} to the copy $\rho X$ of $X$. 

In the example the $\cap$ version of the shape of $X$ is the vector:

\[x^\cap = (x^\cap_{\{f\}},x^\cap_{\{g\}},x^\cap_{\{h\}},x^\cap_{\{f,g\}},x^\cap_{\{f,h\}},x^\cap_{\{g,h\}},x^\cap_{\{f,g,h\}})=(3,3,3,2,0,1,0).\]

Here are some possibilities for the shape $w$ of a $2$-witness $W$ to a copy of $X$: 

A) If $ w_0$ takes its largest possible value of $2|F|=6$ then $ w^\cap_A= w^!_A=0$ for every $A$ with $|A|\geq 2$
\[  w^{\cap}=   (2,2,2,0,0,0,0).     \]
This extreme case appears later as the shape $r$.  Each element of $W_0$ is connected to the $3$ vertices of the face of $X$ over which it lies in Figure \ref{fig:First_Witness_to_Y}.

\begin{figure}[h!]
\centering
\includegraphics[width=1\linewidth]{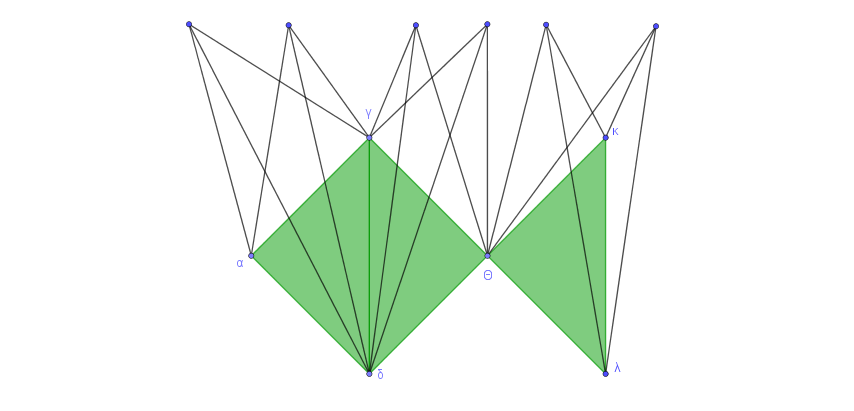}
\caption{First witness to $X$ with $(wx)_0=18$ edges}
\label{fig:First_Witness_to_Y}
\end{figure}

The density of the associated union of three complete bipartite graphs is 

\[ b(x,w)=\frac{(x w)_0}{x_0+w_0}  = \frac{18}{6+6} = \frac{3}{2}.    \]

B) One possibility with $ w_0=5$ has each element of $W_0$ connected to the vertices of the face of $X$ over which it lies in Figure \ref{fig:Second_Witness_to_Y} and 
\[  w^{\cap}=   (2,2,2,1,0,0,0).     \]

\begin{figure}[h!]
\centering
\includegraphics[width=1\linewidth]{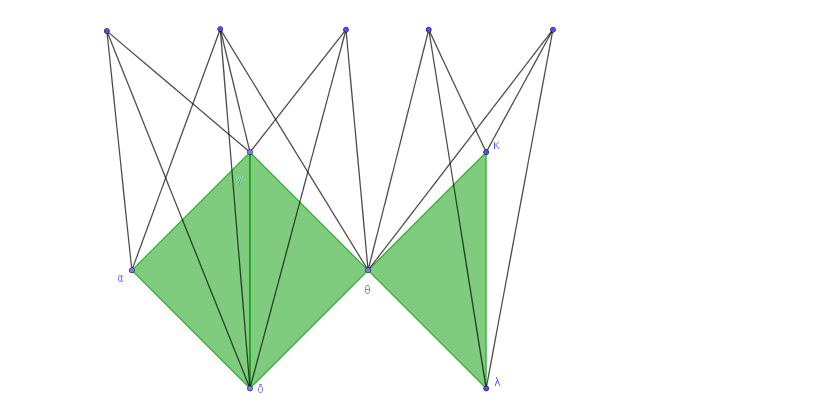}
\caption{Second witness to $X$ with $(wx)_0=16$ edges}
\label{fig:Second_Witness_to_Y}
\end{figure}

The density of this associated union of three complete bipartite graphs is 

\[ b(x,w)=\frac{(x w)_0}{x_0+w_0}  = \frac{16}{6+5} = \frac{16}{11}.     \]

C) If $ w_0=2$ which is the smallest possible value then each element of $W_0$ is connected to every vertex of $X$ as in Figure \ref{fig:Third_Witness_to_Y} and 
\[  w^\cap=   (2,2,2,2,2,2,2).     \]

\begin{figure}[h!]
\centering
\includegraphics[width=1\linewidth]{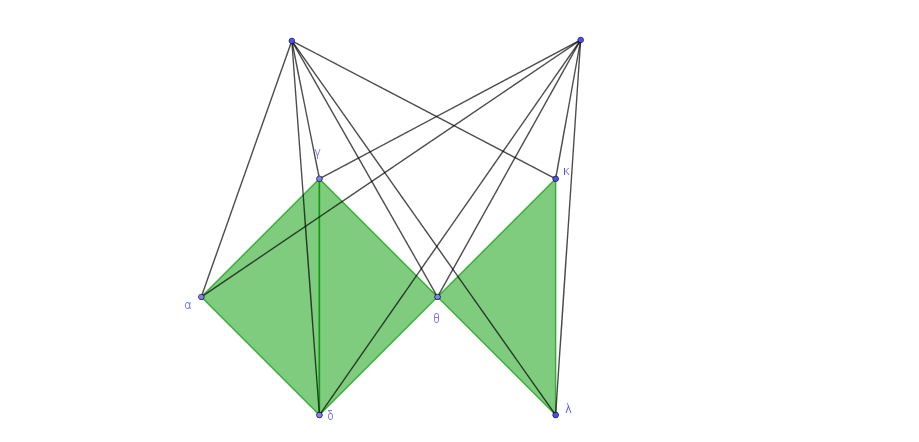}
\caption{Third witness to $X$ with $(wx)_0=12$ edges}
\label{fig:Third_Witness_to_Y}
\end{figure}

The density of this associated complete bipartite graph is 

\[ b(x,w)=\frac{(x w)_0}{x_0+w_0}  = \frac{12}{6+2} = \frac{3}{2}.     \]

\bigskip

\rightline{$\square$}
\bigskip

\begin{thm}  If $X$ is a finite simplicial complex and $m\geq 1$ then the $m$-density of $X$ is a threshold for the appearance of $X$ in $\Gamma_m$.  
\end{thm}

\begin{proof} Write $F=\bigxf$.  For an $m$-pure $F$-set $W$ with shape $w$, by Theorem 5.3 of \cite{FK} above,  the formula
\[\max_{ Z\subseteq X, V\subseteq W}\,  \frac{(z v)_0}{z_0+v_0}       \]
gives a threshold for $H(G,X,W)$ to be nonempty in $\Gamma_m$.  
The threshold for the appearance of $X$ then only requires selecting the $W$ for which this is minimal.  
\end{proof}

Figures \ref{fig:Asymptotics1}--\ref{fig:Asymptotics3} dispaly the average number of copies of $X$ from Examples 1 and 2 observed in 10 draws from $\Gamma_{n,m,p}$ using various values of the
relevant parameters.

\begin{figure}[h!]
\centering
\includegraphics[width=0.5\linewidth]{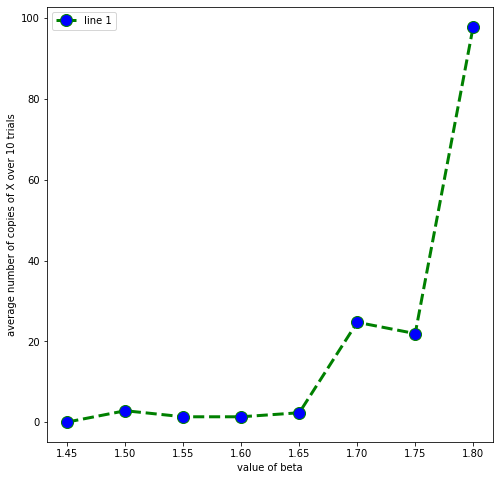}
\caption{Average number of copies of $X$  from Example 1 in $\Gamma_{n,m,p}$, where $n=50$,  $m=2$,   $b\approx 1.4$, $\beta$ is in \{1.45--1.8\} }
\label{fig:Asymptotics1}
\end{figure}

\begin{figure}[h!]
\centering
\includegraphics[width=0.5\linewidth]{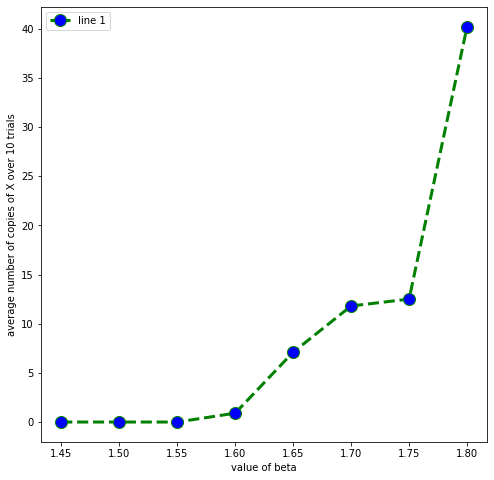}
\caption{Average number of copies of $X$  from Example 2 in $\Gamma_{n,m,p}$, where $n=50$,  $m=2$,   $b\approx 1.4$, $\beta$ is in \{1.45--1.8\} }
\label{fig:Asymptotics2}
\end{figure}

\begin{figure}[h!]
 \centering
 \includegraphics[width=0.5\linewidth]{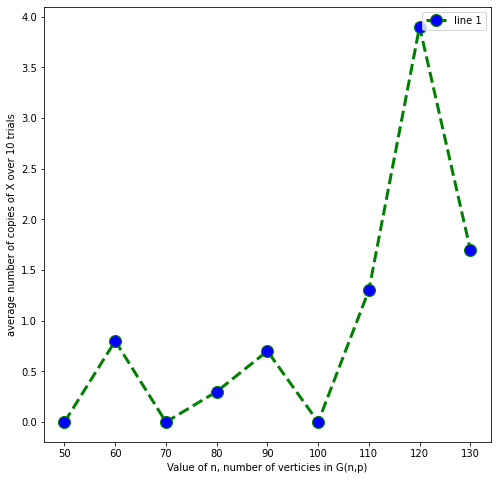}
\caption{Average number of copies of $X$  from Example 1 in $\Gamma_{n,m,p}$, where  $m=4$,   $b\approx 2$, $\beta=2.2$, and  $n$ is in \{50--130\} }
\label{fig:Asymptotics3}
\end{figure}

Write $r$ for the $m$-pure $F$-shape with $r^{\cap}_{A}=0$ for every $A\subseteq F$ with $|A|\geq 2$ so $r_.=r^{\cap}_{\{f\}}=m$ and $r_0=m|F|$ and note that if $x$ is also a pure $F$-shape then 
\[ b(x,r)=\frac{(xr)_0}{x_0+r_0}=\frac{\bar{x}}{\frac{x_0}{m|F|}+1}\]  

\begin{lemma} If $X$ is a finite $(k-1)$-pure simplicial complex with $\phi$ facets and $m>kx_0\phi$ then any $m$-pure $\bigxf$-shape $w\not=r$ has $b(x,w)>b(x,r)$.  
\end{lemma}
\begin{proof} Write $F=\bigxf$ and without loss of generality, assume that $\phi\geq 2$.  Since $w\not=r$ and $r^!_A=0$ for every $|A|\geq 2$ there is some $A\subseteq F$ with $|A|\geq 2$ and $w^!_A\geq 1$.  Fix such a set $A$ and write $v$ for the $F$-shape with $v^!_A=1$, every $a\in A$ has $v^!_{\{a\}}=-1$ and otherwise $v^!_{B}=0$. Hence if $X$ is the complex from Example 1,
$v$ is the $F$-shape described by Figure  \ref{fig:V_Venn_diagram}.

\begin{figure}[h!]
\centering
\includegraphics[width=0.8\linewidth]{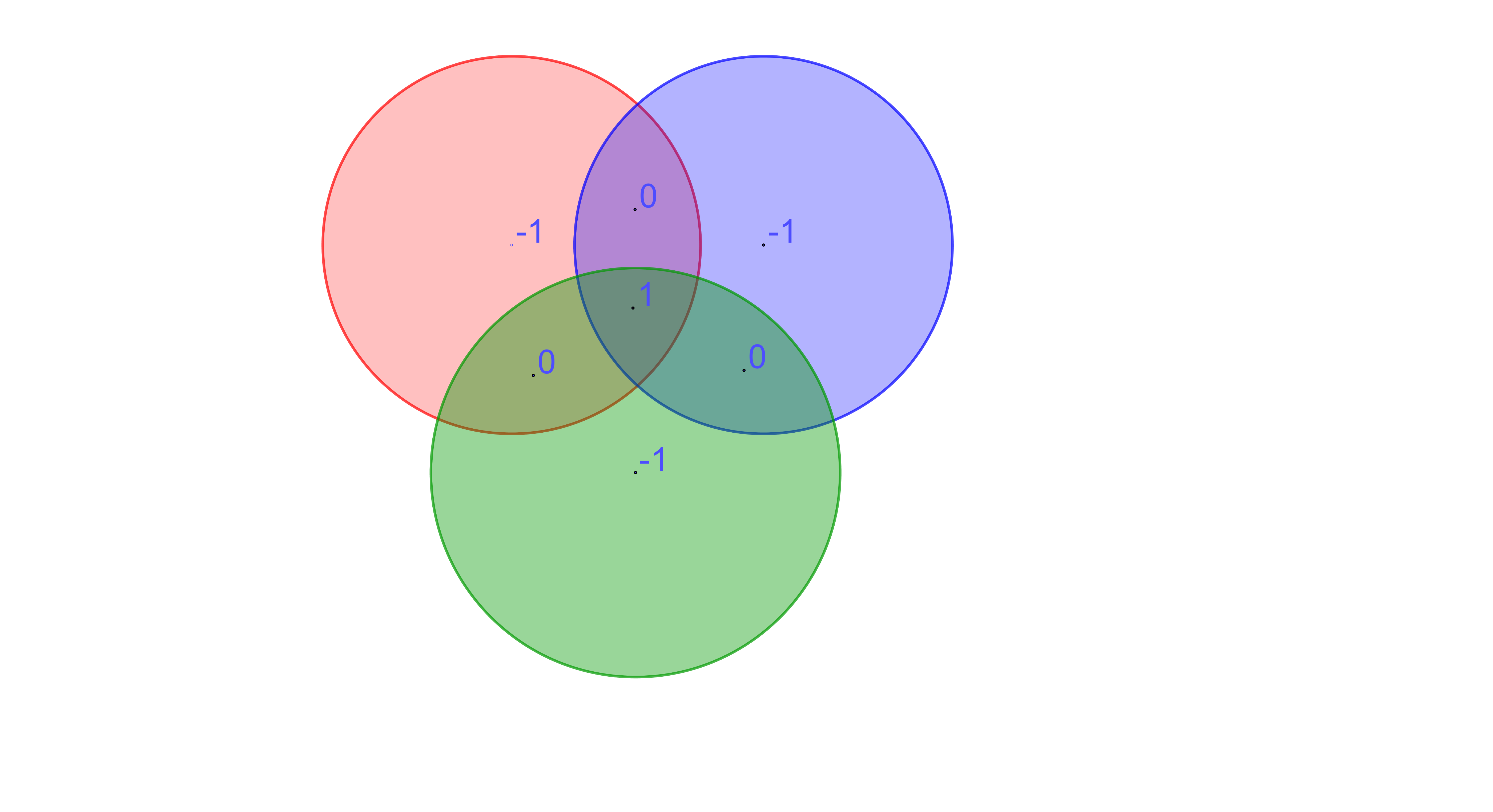}
\caption{The $F$-shape $v$}
\label{fig:V_Venn_diagram}
\end{figure}

Hence if  $u=w-v$ is another $m$-pure $F$-shape it suffices to check that $b(x,w)>b(x,u)$.  
Note that $v^\cap_B=1$ if $B\subseteq A$ and $|B|\geq 2$ and $v^\cap_B=0$ otherwise and that $v_0=1-|A|$ while $(xv)_0=x^\cup_A-|A|k$.  
The last equality follows from the fact that $v^{\cap} = (0,\dots,0,1,1,\dots 1)$ with zeros in the first $|A|$ slots and ones elsewhere. 

  Compute: 
  
 \begin{align*}
 [b(x,w)-b(x,u)]&[(x_0+w_0)(x_0+u_0)]\\
 &=(xw)_0(x_0+u_0)-(xu)_0(x_0+w_0)\\
 &=(xv)_0(x_0+w_0)-(xw)_0v_0\\
 &=(x^\cup_A-|A|k)(x_0+w_0)+(xw)_0(|A|-1)\\
 &=(x^\cup_A-k)w_0 + (|A|-1)[(xw)_0-kw_0] - (|A|k-x^\cup_A)x_0\\
 &> m +0 - \phi k x_0\\ 
 &\geq 0.
\end{align*}

For the strict inequality $x^\cup_A$ is the number of vertices in a union of at least two $(k-1)$-faces and hence at least $k+1$.  Since $w$ is an $m$-pure $w_0$ is at least $m$ so the first term is also.  The second term is positive since $(xw)_0$ is the number of edges connecting a witness of shape $w$ to a copy of $X$ and $w_0$ is the number of vertices in the witness each of which is contained in at least $k$ edges and since $w\not=r$ some vertex is contained in more than $k$ edges.  For the third term $x^\cup_A$ is nonnegative so discard it and $|A|$ is at most $\phi$.  
  
  \end{proof}

\begin{lemma} 
The property of having every size $k$ subset of the vertices as a face has $k$ as a threshold in $\Gamma_m$.
  \end{lemma}
 
\begin{proof}  Write $\Delta$ for the $(k-1)$-pure simplicial complex which is just a single simplex and $\Delta_0=[k]$.  Consider a complex $K$ drawn from $\Gamma_{n,m,p}$ with $p=n^{\frac{-1}{b}}$.

  Consider the number of simplex witnesses $N=|\{\rho\in Z_K\Delta|(\forall i\leq k)\rho i=i\}|$ and compute
\begin{align*}
\log_n \Bbb E(N) &\le \log_n \left[ {n \choose  m}  p^{mk}   \right] \le \log_n \left[ n^m  n^{\frac{-1}{b} mk}   \right] \\
                          &\le  \log_n \left[ n^{m  -\frac{1}{b} mk}   \right] \le m\left(1-\frac{k}{b}\right).
\end{align*} 
The last expression is less than zero for $b<k$, and hence the expected number itself has limit zero as $n$ goes to infinity.
We then apply the First Moment Method (see \cite{FK}, Lemma 22.2): If $X$ is a nonnegative integer
valued random variable, then 
\[  \PP(X>0) \le \EE X.   \]
We conclude that the probability that a given set of $k$ vertices is a face is aas equal to zero giving one of the threshold directions.

For the other direction take $b>k$.  For each $f\in {[n]\choose k}$ the distribution of the number of common neighbors is a binomial variable $X_f\sim B_{n-k}$ with probability $p^k=n^{\frac{-k}{b}}$ and $n-k$ samples.  If $f\cap g=\emptyset $ then $X_f$ and $X_g$ are independent and they are positively correlated otherwise.  Hence the probability that all ${n \choose k}$ such subsets have at least $m$ common neighbors is at least the ${n \choose k}$ power of $\PP(B_{n-k}\ge  m)$.  By the following argument  this is aas equal to one. 

We apply the following version of  the Bernstein--Chernoff bound (see A. Klenke \cite{AK}, Exercise 5.2.1, pg. 110):
If $X_i,\dots,  X_n$ are i.i.d. Bernoulli variables, and $S_n=X_1+\dots+X_n$ with $\mu =\EE S_n$, then for any $\delta$
\[     \PP[S_n\le (1-\delta)\mu ] \le \exp \left(  - \frac{\delta^2 \mu}{2}               \right)                           \]

Note that $\mu_n=\EE(B_{n-k}) = (n-k) n^{-\frac{k}{b}} =  n^{1-\frac{k}{b}} - k n^{-\frac{k}{b}} $.
Choose $\delta_n$ so that $(1-\delta_n)\mu_n = m-1$. The probability $P_n$ that all $k$ element subsets have at least $m$ neighbors is bounded from below as follows:
\begin{align*}
P_n&\ge \left(1- \exp \left(  - \frac{\delta_n^2\  \mu_n}{2}  \right)    \right)^{n \choose k}\\
      &\ge  1-{n \choose k} \exp \left(  - \frac{\delta_n^2\  \mu_n}{2}  \right)\\
      &\ge  1- n^k \exp \left(  - \frac{\delta_n^2\  \mu_n}{2}  \right)\\
       &\ge 1-\exp \left( k\ln n-  \left[1-\frac{m}{\mu_n} +\frac{1}{\mu_n}\right]^2 \frac{\mu_n}{2}      \right)\\
\end{align*}

The last bound has limit 1 as 
\[     \lim_{n\to\infty} \frac {\mu_n}{\ln n}  = \infty      \]

\end{proof}  
 
\begin{lemma} 
The property of having some size $k$ subset of the vertices as a face has $\frac{mk}{m+k}$ as a threshold in $\Gamma_m$.
\end{lemma}
  
\begin{proof} First consider the case $b< \frac{mk}{m+k}$.  
If $K=N_mG$ is drawn from $\Gamma_{n,m,p}$ with $p$ as above write $M=z_GK_{m,k}$ for the number of $m$-witnesses to $(k-1)$-simplices and compute
\begin{align*}
\log_n \Bbb E(M) &= \log_n \left[ {n \choose  k}{{n-k} \choose  m}  p^{mk}   \right]\\
                           & = \log_n \left[ \frac{n!}{k!(n-k)!}\frac{(n-k)!}{(n-k-m)!m!}  p^{mk}   \right] \\
                          &\le  \log_n \left[ n^{k+m  -\frac{1}{b} mk}   \right] - log_n(k! m!)\\
                          & = -\varepsilon  - \log_n(k! m!)
\end{align*} 
where $\varepsilon = \frac{mk}{b} -  (k+m) >0$ so $M$ is aas zero and using the first moment method as above there are aas no $k$ faces in $K$.  

Next consider the case $b>\frac{mk}{m+k}$.
Here we apply the Second Moment Method (see e.g. Lemma 22.5 in \cite{FK}): If $X$ is a nonnegative integer valued random variable, then
\[  \PP(X=0)) \le \frac{\var X}{E(X^2)} = 1 - \frac{(\EE X)^2}{\EE (X^2).} 
\]
We apply this to $M$ to obtain a lower bound on $\PP(M>0)=\PP(M\ge 1)$:
\[          \frac{(\EE M)^2}{\EE (M^2)}   \le   \PP(M>0)                           \]
and then show that
\[ \lim_n\frac{(\EE M)^2}{\EE (M^2)}=1.\]
This is achieved by writing $M^2=\sum_aM_a^2$ a finite sum with the number of terms independent of $n$ and then computing that the limit as $n$ grows without bound of

\begin{equation} 
        \log_n     \frac{(\EE M)^2}{\EE (M_a^2)} = 2 \log_n   \EE M -   \log_n   \EE (M_a^2) \label{difference}   
\end{equation}
is zero for one choice of $a$ and at positive for each of the others.

Minor modifications of the argument in the first part of the proof show that
\[ \lim_n\log_n(\Bbb EM)^2=2\big(k+m-\frac{km}{b}\big). \]

\begin{figure}[h!]
\centering
\includegraphics[width=1\linewidth]{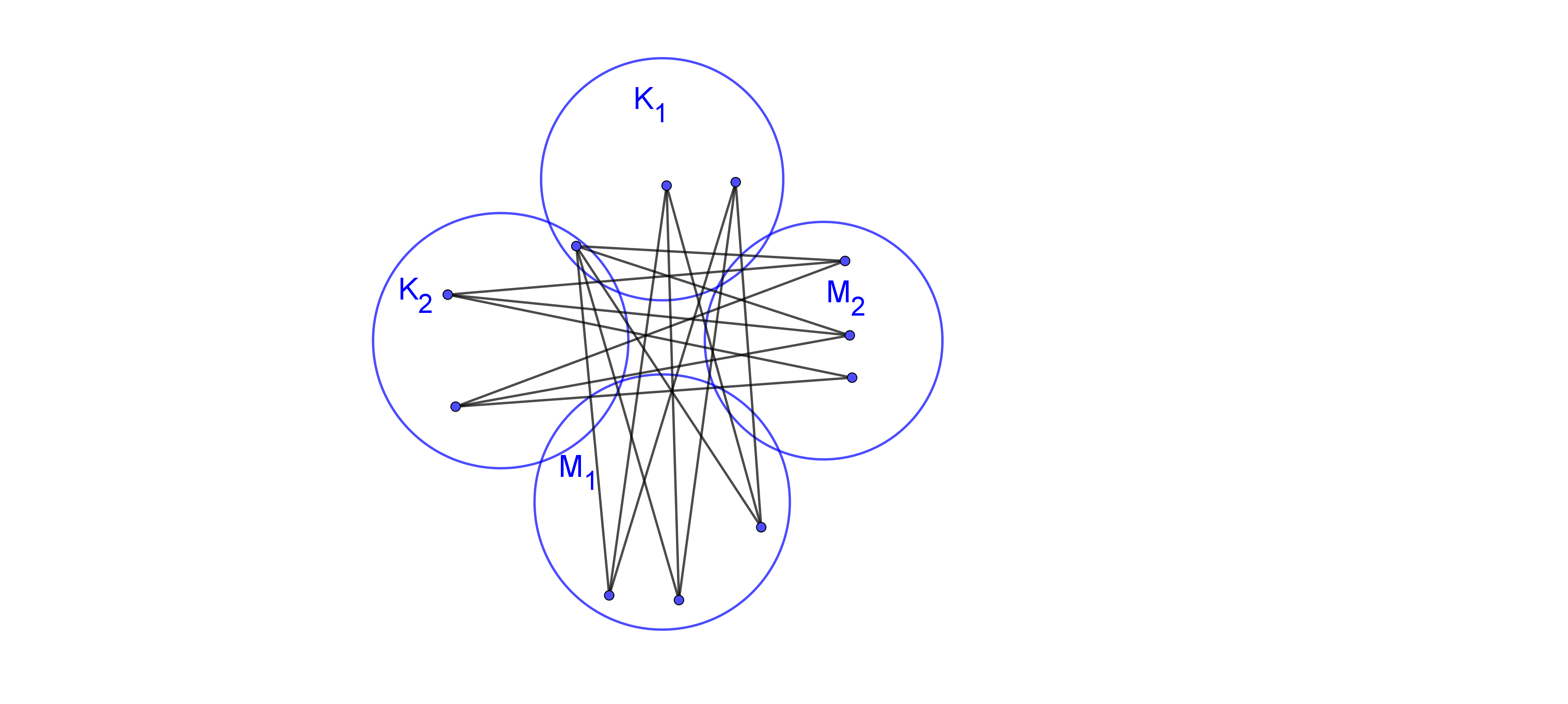}
\caption{}
\end{figure}

Bounds on the second term of (\ref{difference}) are obtained by considering the possible intersection patterns for pairs of $k$-sets $K_1$ and $K_2$ in ${[n]\choose k}$ with $m$-witnesses $M_1$ and $M_2$  in ${[n]\choose m}$ and is indexed by four nonnegative parameters 

\begin{equation*}
\begin{split}
&a_{kk}=|K_1\cap K_2|\\
&a_{km}=|K_1\cap M_2|\\
&a_{mk}=|K_2\cap M_1|\\
&a_{mm}=|M_1\cap M_2|\\
\end{split}
\hskip 1.4 cm
\begin{split}
&a_{kk} + a_{km} \le k\\
&a_{kk} + a_{mk} \le k\\
&a_{mm} + a_{km} \le m\\
&a_{mm} + a_{mk} \le m\\
\end{split}
\end{equation*}

It is then possible to compute $\Bbb E(M^2)$ as a sum over the possible $a_{..}$ values and check that the Expression (\ref{difference}) is positive.  Specifically $\Bbb EM^2$ is the sum over finitely many quadruples $a=\{a_{..}\}$ of $\Bbb EM^2_a$ and $\lim_n\log_n\Bbb EM^2_a=2(k+m-\frac{km}{b})-(a_{kk}+a_{mm}+a_{mk}+a_{km}-\frac{a_{kk}a_{mm}+a_{mk}a_{km}}{b})$.  Since the number of choices for $a$ is a function of $k$ and $m$ independent of $n$, it suffices to check that $(a_{kk}+a_{mm}+a_{mk}+a_{km}-\frac{a_{kk}a_{mm}+a_{mk}a_{km}}{b})>0$ if $a\not=0$ which is an easy check if $b>\frac{mk}{m+k}$.  Specifically, we can apply 1-variable calculus to the function $f(x)= x + y - \frac{xy}{b}$, where
 $y$ is held fixed in $[0,m]$ and $x$  ranges over $[0,k]$, to see that the function takes positive values on $(0,k)$.

It follows that the  order of magnitude (as a power of $n$) of the expectation of $M^2$ does not exceed that of  the square of the expectation of $M$.
It follows that the limit of Expression (\ref{difference}) is equal to zero and  aas $\PP(M>0)=1$.

\end{proof}

\begin{cor}
  If $k<\sqrt{2m+1}$ there is an interval of $\beta$ values for which $\Gamma_{m,n,p}$ with $p=n^{\frac{-1}{\beta}}$ is aas in $Y_{n,k-1}$ but they are not all the same, while if $k^2+k<m$ and $\max\{k,\frac{mk}{m+k}\}< \beta < \min\{k+1,\frac{mk+m}{m+k+1}\}$ then all such complexes are aas always the complex with every set of vertices of size $k$ a face and none of size $k+1$.
  \end{cor}

\begin{thm}  If $X$ is a finite $(k-1)$-pure simplicial complex with $\phi$ facets, $m>kx_0\phi$ and $k>\beta>\frac{km\phi}{x_0+m\phi}$ then with $p=n^{\frac{-1}{\beta}}$ and $q=n^{m(1-\frac{k}{\beta})}$ there is $\lim_{n\rightarrow\infty}\frac{\EE_{K\in\Gamma_{n,m,p}}(z_K X)}{\EE_{K\in Y_{n,k,q}}(z_K X)}=1$.  
\end{thm}

Write $R$ for the extreme $m$-pure $F$-set of shape $r$ with every $r^\cap_A=0$ if $A\in {F\choose 2}$.  

\begin{proof}
Write $F$ for the facets of $X$ and let $a=1-\frac{k}{\beta}$.   Note that

\begin{align*}\lim_{n\rightarrow\infty}\log_n\EE_{K\in Y_{n,k,q}}(z_KX) &= \lim_{n\rightarrow\infty}\log_n {n \choose x_0} q^{\phi} \\
                                                                                                              & = \lim_{n\rightarrow\infty}\log_n {n \choose x_0} n^{m \phi a} = x_0+am\phi
\end{align*}

For the lower bound there are enough copies of $X$ with witnesses of shape $r$.  If $G$ is a graph write 
\begin{align*}[M(G, X)=&\{\rho\in H(G, R, X)| \\
                  &\forall (f\in F, g\in V G-\rho R^\cup_{\{f\}})\exists(h\in \rho X^\cup_{\{f\}}, (g,h)\not\in E G)\}
\end{align*} 
for the minimal $R$-witnesses to $X$ in $N_mG$ and $m(G,X)=|M(G,X)|$.  
Note that if $\rho, \rho'\in M(G,X)$ with $\rho X_0=\rho' X_0$ are witnesses to the same copy of $X$ up to symmetry and $f\in F$ then the definition of $M$ implies that $\rho R^{\cup}_{\{f\}}=\rho' R^{\cup}_{\{f\}}$.  Thus $\rho$ and $\rho'$ differ by one of at most $c=(m!k!)^\phi\phi !$ automorphisms and $m(G,X)\leq cz_{N_mG}X$ with $c$ independent of $|VG|$.  
In the  Erd\H os--R\'enyi setting $|VG|=n$ so for any of the $\frac{n!}{(n-x_0-m\phi)!}\sim n^{x_0+m\phi}$ injective maps $\rho:R_0\cup X_0\rightarrow V G$ we have 
\[\PP_{G\in G(n,p)}(\rho \in M(G,X))=p^{(rx)_0}(1-p^k)^{(n-x_0-m\phi)\phi}\]
Since $a=1-\frac{k}{\beta}<0$ by the choice of $\beta$, after some simplification this gives  
\begin{align*}\lim_{n\rightarrow\infty}\log_n\EE_{K\in\Gamma_{n,m,p}}(z_K X)&\geq \lim_{n\rightarrow\infty} x_0+m\phi-\frac{mk\phi }{\beta} -2\phi\ln^{-1}(n)n^a \\
                                                                                                                           & = x_0+am\phi.
\end{align*}
The middle inequality uses the approximation $1-c\geq e^{-2c}$ if $c<\frac{3}{4}$ with $c=p^k=n^\frac{-k}{\beta}$.  

For the upper bound note that the number of $m$-pure $F$-shapes $w$ is bounded independent of $n$ (by $(m+1)^{2^{\phi}}$ for instance), let $\Omega$ be the finite set of all such shapes. 
First, we will establish that 
\[       w_0 - \frac{(xw)_0}{\beta} < a m \phi       \]
Note that 
\[       w_0 - \frac{(xw)_0}{\beta} < \left(1-\frac{k}{\beta}\right) m \phi   \iff   \frac{km\phi}{\beta} - \frac{(xw)_0}{\beta} < m \phi  - w_0\]
Since by our assumptions on $\beta$, we have
\[  \frac{km\phi - (xw)_0}{\beta} <    \frac{km\phi - (xw)_0}{\frac{km\phi}{x_0+m\phi}}  \]
It is enough to establish that
\[   \frac{km\phi - (xw)_0}{\frac{km\phi}{x_0+m\phi}} <      m \phi  - w_0     \]
The following inequalities are equivalent
\begin{align*}
[km\phi - (xw)_0][x_0+m\phi]  &<  km\phi (m \phi  - w_0)  \\
x_0 k m \phi  -  x_0(xw)_0 - m \phi (xw)_0 &< - w_0 k m \phi  \\
(x_0 + w_0) k m \phi &<   (x_0 + m \phi) (xw)_0 \\
b(x,r)=\frac{ k m \phi}{(x_0 + m \phi)} &<    \frac{(xw)_0}{(x_0 + w_0)} = b(x,w)
\end{align*}

By  Lemma  3 if $r\not= w$,  $b(x,r)<b(x,w)$, hence the final inequality holds.

Since the expected number  $\EE_{G\in G(n,p)}(h(G,w,x))$  of copies of $X$ in $N_mG$ together with a witness of shape $w$ is approximately $n^{x_0+w_0-\frac{(xw)_0}{\beta}}$, we have

\begin{align*}
\lim_{n\rightarrow\infty}\log_n\EE_{K\in\Gamma_{n,m,p}}(z_K X) &=  \lim_{n\rightarrow\infty}\log_n  \sum_{w\in \Omega} \EE_{G\in G(n,p)}(h(G,w, X))\\
                                                                                                    &=  \lim_{n\rightarrow\infty}\log_n  \sum_{w\in \Omega}  n^{x_0+w_0 - \frac{(xw)_0}{\beta}}\\
                                                                                                     &\le  \lim_{n\rightarrow\infty}\log_n  \sum_{w\in \Omega}  n^{x_0+a m \phi}\\
                                                                                                     &=  \lim_{n\rightarrow\infty}\log_n  \left( |\Omega| \cdot   n^{x_0+a m \phi}\right)\\
                                                                                                     &= x_0+am\phi\\
\end{align*}

\end{proof}

\begin{conj}  For every $k$ here is a sequence $n_m$ for which the total variation distance between the $\Gamma_{n,m,p}$ and $Y_{n,k,q}$ distributions tends to zero as $m$ tends to infinity if $n=n_m$, $\beta=k-\frac{1}{n_m}$, $p=n^{\frac{-1}{\beta}}$ and $q=n^{\frac{-km}{\beta}}$.  
  \end{conj}
  
\begin{conj}  There is a choice of $m$, $\beta$ and a positive $k$-pure shape $x$ for which if $p=n^{\frac{-1}{\beta}}$ then the support of $\Gamma_{n,m,p}$ is aas in $Y_{n,k}$ but if also $q=n^{\frac{-km}{\beta}}$ then $\lim_{n\rightarrow\infty}\frac{\EE_{K\in\Gamma_{n,m,p}}(z_Kx)}{\EE_{K\in Y_{n,k,q}}(z_Kx)}\not=1$.  \end{conj}

\medskip

In an effort to find an example for the second conjecture or disprove it consider the following reduction.  If $W$ is an $m$- witness for a $(k-1)$-pure complex $X$ write 
\begin{align*}
\begin{split}
\bar w &= m, \\
\bar x &= k , \\
x_w &= \frac{x_0\bar w}{w_0\bar x},\\
\pi^w_x &= \frac{(xw)_0}{x_0\bar w}\geq 1,\\
\pi^x_w &= \frac{(xw)_0}{w_0\bar x}\geq 1,\\
\end{split}
\hskip 2 cm
\begin{split}
\phi &= x_F,\\ 
\phi_x &= \frac{\phi \bar x}{x_0}\geq 1,  \\
\phi_w &= \frac{\phi \bar w}{w_0}\geq 1,\\
 b &= \frac{(xw)_0}{x_0+w_0}=\frac{\bar x\bar w}{\bar x(\pi^w_x)^{-1}+\bar w(\pi^x_w)^{-1}}. 
\end{split}
\end{align*}
The above lemmas imply that there is a choice of $\beta$ in the conjecture if 
\begin{align*}
\bar x-1<b \qquad\textrm{ (so all $k-2$ faces occur aas), }\\
b<\frac{\bar w(\bar x+1)}{\bar w+\bar x+1} \qquad\textrm{(so no $k$ faces occur aas) and }\\ 
b<\frac{\phi \bar x\bar w}{x_0+\phi \bar w} \qquad\textrm{ (so $X$ does not occur in $Y_{n,k,q}$).} 
\end{align*}
 Substituting the definition of $b$, cross multiplying and collecting terms involving $\bar w$ in these three inequalities allows them to be rewritten as
\[\frac{x_w(\bar x-1)}{1+\bar x(\pi^x_w-1)}<\frac{\bar w}{\bar x},\]
\[\frac{(\pi^x_w-x_w)(\bar x+1)}{1-\bar x(\pi^x_w-1)}<\frac{\bar w}{\bar x},\]
\[\frac{\bar w}{\bar x}<\frac{\phi_w-\pi^x_w}{\phi_x(\pi^x_w-1)}\] respectively.
Finally the first and third and second and third yield inequalities by ignoring the intervening $\frac{\bar w}{\bar x}$ and again crossmultiplying and collecting $\bar x$ terms yields equivalent inequalities
\begin{align}
\bar x&<\frac{\phi_w-1}{\pi^x_w-1}, \label{Conjecture2_inequality1} \\
\bar x&<\frac{1}{\pi^x_w-1}(1-\phi_w\frac{\pi^w_x-1}{\phi_x-1})\label{Conjecture2_inequality2} 
\end{align}
respectively.  Perhaps this formulation is easier to work with.  

\medskip

{\bf Example 2, continued.} Returning to the complex $X$ of shape $x$ with facets $F=\{f,g,h\}$: 

\[       f=\{\alpha,\gamma,\delta \},   g=\{ \gamma,\delta, \theta  \}, h=\{\theta, \kappa, \lambda\}                   \]
we have $x_0=6$, $\bar{x}=3$, $x_F=3$ and $\phi_x=\frac{3}{2}$.  For the 2-witnesses W mentioned earlier, the above parameters take the following form.

A) We have $(wx)_0=18$, $\phi_w=1$, $\pi^x_w=1$, $\pi^w_x=\frac{3}{2}$ and 
 the inequalities (\ref{Conjecture2_inequality1}) and (\ref{Conjecture2_inequality2}) 
 that would be needed for a counterexample do not apply since $W=R$.  In both cases the left hand side equals 3, and the right hand sides are fractions with zero in both the numerator and denominator.   

B) We have $(wx)_0=16$, $\phi_w=\frac{6}{5}$, $\pi^x_w=\frac{16}{15}$, $\pi^w_x=\frac{4}{3}$ and the inequalities (\ref{Conjecture2_inequality1}) and (\ref{Conjecture2_inequality2}) that would be needed for a counterexample are almost satisfied (both turn out to be $3<3$) .   

C) We have $(wx)_0=12$, $\phi_w=3$, $\pi^x_w=2$, $\pi^w_x=1$ and 
 the inequalities (\ref{Conjecture2_inequality1}) and (\ref{Conjecture2_inequality2}) 
 that would be needed for a counterexample are $3<2$ and $3<1$ respectively.

\medskip

A simple case to study is that in which $X$ has only two facets which each have $k$ vertices and share $\ell$ of these.

In this case if $m<\frac{\ell(2k-\ell)}{2(k-\ell)}$ then $\Gamma_{n,m,n^{\frac{-1}{\beta}}}$ aas contains all $k-1$ faces if $\beta> \frac{(k+1)m}{k+1+m}$ and aas contains no pair of $k-1$ faces which intersect in $\ell$ vertices if $\beta< \frac{(k+1)m}{k+1+m}$.  If on the other hand $m>\frac{\ell(2k-\ell)}{2(k-\ell)}$ there is a third phenomenon.  If $\frac{2km}{2k+2m-\ell}<\beta<\frac{(k+1)m}{k+1+m}$ then $\Gamma$ aas contains pairs of $k-1$ faces which share $\ell$ vertices but no $k$ faces.  In this case the expected number of such pairs is approximately $n^{2k+2m-\ell-\frac{1}{\beta}2km}$.


\section{Computer Simulations}

In this section we briefly summarize the results of computer simulations which are consistent with the results presented in the earlier section. 
Using the Networkx and Gudhi Python libraries, we have implemented the Erd\H os--R\'enyi graphs and the m-neighbor construction. 
For $n=150$ vertices, the probability $p=0.2$ and the values of $m$ in the set $\{2,4,8,12\}$ we have obtained random complexes whoses number of 
simplices is presented in the following table.

\begin{tabular}{|c |c |c| c| c| c| c|}
\hline
Number of neighbors     & m=1     &  m=2     &  m=4     &   m=8    &  m=12         \\
                                    &             &               &              &             &                    \\
\hline
                                    &             &               &              &             &                      \\
Value of     $t$             &     3         &     3       &   2         &    2        &       2           \\
                                    &             &               &              &             &                       \\
\hline
                                    &             &               &              &             &                      \\
Number of Simplices     &  11160   & 11097     &     150   &    150   &         150        \\
in dimension $t-2$        &             &               &              &             &                       \\
\hline
                                    &             &               &              &             &                      \\
Number of Simplices     &  389580  &  219430 &     9515  &   2882     &     207          \\
in dimension $t-1$        &             &               &              &             &                       \\
\hline
                                    &             &               &              &             &                      \\
Number of Simplices     &      0       &        0     &    0       &       0    &         0         \\
in dimension $t$            &             &               &              &             &                       \\
\hline
Ratio of simplices in      &             &               &              &             &                      \\
dimension $t-2$ to       &  0.999    &    0.993  &    1        &       1   &           1           \\
the maximal possible    &             &               &              &             &                       \\
\hline
Ratio of simplices in      &             &               &              &             &                      \\
dimension $t-1$ to        &   0.707   &    0.398  &   0.851  &  0.257   &   0.0185           \\
the maximal possible     &             &               &              &             &                       \\
\hline
Estimate of  $q$            &             &               &              &             &                      \\
via the binomial            &    0.693 &  0.329   &    0.847    &  0.242  &       0.0162       \\
distribution                    &             &               &              &             &                       \\
\hline
\end{tabular}

\section{Proofs}

{\bf Hoeffding's Inequalities.}\ 
{\it If $B\sim\bin n q$ then:
\begin{itemize}
\item $\PP(B\leq m)\leq \exp[-\frac{2}{n}(nq-m)^2]$ if $m<nq$ and 
\item $\PP(B\geq m)\leq \exp[-\frac{2}{n}(nq-m)^2]$ if $m>nq$.  
\end{itemize}
}
\medskip

\begin{proof}
We start with the original theorem of Hoeffding:

\medskip
{\it Theorem 2 \cite{H}:} If $X_1,\dots,X_n$ are independent, $a_i\le X_i\le b_i$ for $i=1,\dots,n$, and $\mu=\EE\bar X$ with $\bar X=\frac{1}{n}\sum_{i=1}^n X_i$ then for $t>0$
\[         \PP\{\bar X - \mu \ge t \} \le \e^{-2n^2t^2/\sum_{i=1}^n (b_i-a_i)^2}.                  \]

If we assume that the $X_i \sim \bin 1  q$  are copies of the Bernoulli random variable with success probability $q$, then
$B=n \bar X$ has the binomial distribution $B \sim \bin n q$. The above inequality can be written in the form:
\[         \PP\{B - nq \ge nt \} \le \exp\left[-2n t^2   \right].               \]

If $m>n\mu=nq$ and we set $c=\frac{m-n\mu}{n}$ then:

\begin{align*}
\PP(B \ge m) &= \PP(B \ge n\mu+nc) \\
                                &= \PP(B - n\mu \ge nc) \\
                                & \le \exp\left[-2n c^2 \right] \\
                                & \le \exp\left[-2n \left( \frac{m-n\mu}{n} \right)^2  \right]\\
                                & \le \exp\left[-\frac{2}{n} ( m -nq) ^2  \right].\\
\end{align*}
This finishes the proof of the second inequality above.

Next, let $m<n\mu$. We have

\begin{align*}
\PP(B \le m) &= \sum_{i=0}^m {n\choose i} q^i (1-q)^{n-i}\qquad ( \textrm{Set}\quad j= n-i\ ) \\
                              &= \sum_{j=n-m}^n {n\choose n-j} q^{n-j} (1-q)^{j} \\
                               &= \sum_{j=n-m}^n {n\choose j}  (1-q)^{j} q^{n-j}\\
                               &=\PP(\hat B \ge n- m), \qquad\textrm{where}\quad \hat B \sim \bin n {1-q}.
\end{align*}

Note that since $m<n\mu=nq$, we have $n-m>n-nq=n(1-q)$ which is the expected value of $\hat B$.  Applying the second inequality proved above, we have

\begin{align*}
\PP(\hat B  \ge n- m) &\le \exp\left[-2n \left( \frac{n-m}{n} -(1-q)\right) ^2 \right] \\
                                       &= \exp\left[-\frac{2}{n} ( nq - m) ^2  \right].
\end{align*}

This finishes the proof of the first inequality.

\end{proof}


\end{document}